\documentclass[11pt, reqno]{amsart}

\usepackage{amsmath,amssymb,amsthm, epsfig}
\usepackage{hyperref}
\usepackage{amsmath}
\usepackage{amssymb}
\usepackage{color}
\usepackage{ulem}
\usepackage{epsfig}
\usepackage[mathscr]{eucal}
\usepackage[latin1]{inputenc}

\newtheorem{theorem}{Theorem}

\newtheorem{proposition}{Proposition}

\date{}
\numberwithin{equation}{section}
\numberwithin{theorem}{section}
\numberwithin{lemma}{section}
\numberwithin{corollary}{section}
\numberwithin{remark}{section} 
\numberwithin{proposition}{section}
\numberwithin{definition}{section}

\def \Div {\mathrm{div}} 

\def \R {\mathbb{R}}

\begin{document}

\title[An optimal Liouville theorem for the porous medium equation]{An optimal Liouville theorem for the porous medium equation}

\author[D.J. Ara\'ujo]{Dami\~ao J. Ara\'ujo}
\address{Department of Mathematics, Federal University of Para\'iba, 58059-900, Jo\~ao Pessoa, PB, Brazil}
\email{araujo@mat.ufpb.br}

\author[R. Teymurazyan]{Rafayel Teymurazyan}
\address{CMUC, Department of Mathematics, University of Coimbra, 3001-501 Coimbra, Portugal}
\email{rafayel@utexas.edu}

\begin{abstract}
	Under a sharp asymptotic growth condition at infinity, we prove a Liouville type theorem for the inhomogeneous porous medium equation, provided it stays universally close to the heat equation. Additionally, for the homogeneous equation, we show that for the conclusion to hold, it is enough to assume the sharp asymptotic growth at infinity only in the space variable. The results are optimal, meaning that the growth condition at infinity cannot be weakened.

\bigskip


\bigskip


\end{abstract}

\maketitle

\section{Introduction}\label{s1}
As it is well known, bounded harmonic function in the whole space $\R^n$ must be a constant (Liouville theorem). This is true for parabolic equations as well, namely, in $\R^n\times\R_-$ bounded solutions of the heat equation are constant, \cite{W75}. However, there is a key difference between elliptic and parabolic equations. For instance, entire harmonic functions, which are bounded only from below (or above), are constant, \cite{GT01}, but one sided bound is not enough to make the same conclusion for entire caloric functions, as shows the example of the function $u(x,t)=e^{x+t}$, which solves the heat equation in $\R\times\R$, is bounded from below, but obviously is not a constant. Nevertheless, the absence of the two sided bound can be compensated by a growth condition, as established by Bernstein. He argued that in the plane entire solutions of uniformly elliptic equations, growing sublinearly at infinity, must be constant, \cite{GT01}. A result of similar spirit was established for entire solutions of the heat equation by Hirschman in \cite{H52}. Observe that for the inhomogeneous equation, even in the uniformly elliptic case, one sided bound is not enough to guarantee a Liouville type result, as shows the example of the non-constant function $u(x)=|x|^2\ge0$, which satisfies $\Delta u=2n$ in $\R^n$. Recent advances in geometric analysis, \cite{DGV08,DGV12,TU142}, allowed to obtain Liouville type results for degenerate elliptic and parabolic equations. These results are intrinsically related to the regularity estimates (see, for example, \cite{KSZ07,TU14}). In particular, using sharp H\"older regularity of solutions, in \cite{TU14} the authors show that weak solutions of
$$
u_t-\Div\left(|\nabla u|^{p-2}\nabla u\right)=0,\,\,\,p>2
$$
in $\R^n\times\R$ are constant, provided their growth at infinity is controlled in an intrinsic manner. Local $C^{\alpha,\alpha/2}$ estimates for weak solutions of the inhomogeneous porous medium type equation
\begin{equation}\label{1.1}
	u_t-\Div(u^\gamma\nabla u)=f,
\end{equation}
where $\gamma\ge0$ and $f\in L^\infty$, \cite{DGV08,DGV12}, suggest similar result also for entire solutions of \eqref{1.1}, particularly, for those behaving like $o(|x|^\alpha+|t|^{\frac{\alpha}{^{2}}})$, as $|x|+|t|\to\infty$. Note that \eqref{1.1} admits solutions with faster growth rate, such as, for $\gamma>0$, 
$$
u(x,t)= \left(\frac{1}{\gamma}\right)^{\frac{1}{\gamma}}|x|^{\frac{2}{\gamma}}(-t)^{-\frac{1}{\gamma}} \quad \mbox{ in }\quad\; \mathbb{R}^n \times \mathbb{R}_{-},
$$
indicating that there is an upper bound for $\alpha$ under which Liouville type results hold. 

In this note, using a higher (optimal) regularity estimate obtained recently in \cite{A19}, we prove an optimal Liouville theorem for solutions of \eqref{1.1} in a strip, provided the diffusion parameter $\gamma\ge0$ is small. More precisely, we show that if a solution of \eqref{1.1} vanishes at a point and 
\begin{equation}\label{condition}
	u(x,t) = o\left(|x|+|t|^{\frac{1}{^{\sigma}}}\right)^{\sigma}, \quad \mbox{as}\quad |x|+|t|\to\infty,
\end{equation}
where $\sigma=2(\gamma+1)^{-1}$, then it must vanish at all ``ancient'' times (Theorem \ref{Liouville}). Moreover, our result is optimal, meaning that the growth condition at infinity cannot be weakened, as shows the example of the solution $u(x,t)=|x|^{\frac{2}{\gamma+1}}$, which vanishes at $\{0\}\times\R$, satisfies \eqref{condition} for $\sigma=2(\gamma+1)^{-1}+\delta$, $\delta>0$, and obviously is not identically zero. Furthermore, for the homogeneous equation, we show that for the conclusion to hold, it is enough to assume the sharp asymptotic growth at infinity only in the space variable (Theorem \ref{t3.2}). Our result extends similar results for the heat equation, \cite{H52,W75}, and adjusts the result from \cite[Proposition 16.2]{DGV12}. It also covers equations with obstacle type non-homogeneity (see \cite{BLS15,CS20}), such as
$$
u_t-\Div(u^\gamma\nabla u)=h\chi_{\{u>0\}},
$$
where $h\in L^\infty$ and $\chi_{\{u>0\}}$ is the characteristic function of the set $\{u>0\}$. Unlike \cite{AS10,S09}, the degree of homogeneity in \eqref{1.1} does not depend on the right hand side.

\section{Preliminaries}\label{s2}
One may notice that the classical porous medium equation has the coefficient $(\gamma+1)$ in the divergence term of \eqref{1.1}, however, it does not affect the proofs of the arguments, therefore, to keep things simple, we consider the equation \eqref{1.1}. 

Next, we introduce notations and recall a regularity result for future reference. We start by defining the intrinsic cylinder by
$$
G_r(x,t):=B_r(x)\times\left(-r^{\frac{2}{\gamma+1}}+t,\,t\right],\,\,\,r>0
$$
and the intrinsic norm by
$$
\|(x,t)\|:=|x|+|t|^{\frac{\gamma+1}{2}}.
$$
In what follows, we will use $G_r$ for $G_r(0,0)$. We also define the $T$-level strip by
$$
S_T:=\R^n\times(-\infty, T),\,\,\,T\in\R.
$$
Solutions of \eqref{1.1} are understood in the weak sense (for the precise definition see, for example, \cite{DGV12,V07}). Our results makes use of a higher (optimal) regularity result obtained for solutions of the porous medium equation near the heat equation. More precisely, we use the following regularity result  from \cite[Theorem 2]{A19}. 
\begin{theorem}\label{growth}
	If $u\ge0$ is a weak solution of 
	\begin{equation}\label{nonhomogeneousequation}
		u_t-\Div(u^\gamma\nabla u)=f\,\,\,\textrm{ in }\,\,\,G_1,
	\end{equation}
	with $f\in L^\infty(G_1)$, then there exist $\varepsilon>0$ and $C>0$, depending only on $\|f\|_\infty$ and $n$, such that for $\gamma\in(0,\varepsilon)$ and $(x_0,t_0)\in\partial\{u>0\}\cap G_{1/2}$, in $G_{1/5}(x_0,t_0)$, one has
	$$
	u(x,t)\le C\|u\|_\infty\|(x-x_0,t-t_0)\|^{\frac{2}{\gamma+1}}.
	$$
\end{theorem}
We finish this section with the following result from \cite[Proposition 16.2]{DGV12}, which will later be used to prove that for the homogeneous equation, for the Liouville type result to hold, it is enough to assume the sharp asymptotic growth at infinity only in the space variable.
\begin{proposition}\label{dibenedetto}
	If $u\ge0$ is a weak solution of 
	$$
	u_t-\Div(u^\gamma\nabla u)=0
	$$
	in a strip $S_T$, $T\in\R$, and $\displaystyle\inf_{S_T}u=0$, then
	$$
	\lim_{t\rightarrow-\infty}u(x,t)=0,\,\,\,\text{ for all }\,\,\,x\in\R^n,
	$$
	and the limit is uniform in $x$ ranging over a compact set $Q\in\R^n$ such that $Q\times\{\tau\}$ is included in a $(y,s)$-paraboloid
	$$
	P(y,s):=\left\{(x,t)\in S_T;\,\,t-s\le-\left(\frac{c}{u(y,s)}\right)^\gamma|x-y|^2\right\}
	$$
	for $(y,s)\in S_T$ such that $u(y,s)>0$ for some $\tau<s$.
\end{proposition}

\section{Liouville type theorems}\label{s3}
In this section we prove that if a non-negative weak solution of the inhomogeneous porous medium type equation in a strip vanishes at a point and has a certain intrinsic growth at infinity, then it must vanish at all ``ancient'' times, provided the equation is universally close to the heat equation. Our result is optimal, i.e., the growth condition at infinity cannot be weakened. Furthermore, for the homogeneous equation, we show that for the conclusion to hold, it is enough to assume the asymptotic growth at infinity only in the space variable. Thus, we assume that $\gamma\ge0$ is small enough, so we are in the regularity regime of Theorem \ref{growth}.
\begin{theorem}\label{Liouville}
	Let $\gamma\ge0$ be as in Theorem \ref{growth}, $T\in\R$, and $u\ge0$ be a weak solution of
	$$
	u_t-\Div(u^\gamma\nabla u)=f
	$$
	in $S_T$, where $f\in L^\infty(S_T)$. If $u(x_0,t_0)=0$ for some $(x_0,t_0)\in S_T$ and
	\begin{equation}\label{intrinsicgrowth}
		u(x,t)=o\left(\|(x,t)\|^{\frac{2}{\gamma+1}}\right),\,\,\,\text{ as }\,\,\,\|(x,t)\|\rightarrow\infty, 
	\end{equation}
	then $u\equiv0$ in $\overline{S}_{t_0}$.
\end{theorem}
\begin{proof}
	Set 
	$$
	\sigma:={\frac{2}{\gamma+1}}
	$$ 
	and for $k\in\mathbb{N}$ define	
	$$
	u_k(x,t):=k^{-\sigma}u(kx+x_0,k^\sigma t+t_0)
	$$
	Direct computation shows 
	$$
	(u_k)_t-\Div(u_k^\gamma\nabla u_k)=u_t-\gamma k^{2(1-\sigma)-\sigma(\gamma-1)}u^{\gamma-1}|\nabla u|^2-k^{2-\sigma-\gamma\sigma}u^\gamma\Delta u,
	$$
	also
	$$
	2(1-\sigma)-\sigma(\gamma-1)=0\,\,\,\text{ and }\,\,\,2-\sigma-\gamma\sigma=0,
	$$
	therefore,
	\begin{equation*}
		\begin{split}
	(u_k)_t-\Div(u_k^\gamma\nabla u_k)&=u_t-\gamma u^{\gamma-1}|\nabla u|^2-u^\gamma\Delta u\\
	&=u_t-\Div(u^\gamma\nabla u)\\
	&=f(kx+x_0,k^\sigma t+t_0),
		\end{split}
	\end{equation*}	
	where the right hand side is bounded independent of $k$. Observe also that $u_k(0,0)=0$. We divide the rest of the proof into two steps. 
	
	\textit{Step 1.} We claim that in $G_{1/5}$ there holds
	\begin{equation}\label{bounded}
		\|u_k\|_\infty\rightarrow0,\,\,\,\textrm{ as }\,\,\,k\rightarrow\infty.
	\end{equation}
    Indeed, if that is not the case, then there is a constant $c>0$, a sequence $k_i\rightarrow\infty$ and points $(x_i,t_i)\in G_{1/5}$, such that 
    $$
    u_{k_i}(x_i,t_i)\ge c>0,\,\,\,\textrm{ as }\,\,\,k_i\to\infty.
    $$
    Hence, from the definition of $u_k$ it follows that the sequence $(k_ix_i,k_i^\sigma t_i)$ is unbounded. Therefore, recalling the definition of the intrinsic norm, up to a subsequence, we must have
    $$
    \|(k_ix_i,k_i^\sigma t_i)\|=|k_ix_i|+|k_i^\sigma t_i|^{\frac{1}{\sigma}}=k_i\|(x_i,t_i)\|\to\infty.
    $$
    The latter, combined with \eqref{intrinsicgrowth}, yields
    \begin{equation*}
    	\begin{split}    
    u_{k_i}(x_i,t_i)&=k_i^{-\sigma} u(k_ix_i+x_0,k_i^\sigma t_i+t_0)\\
    &=k_i^{-\sigma}o\left(\|(k_ix_i,k_i^\sigma t_i)\|^\sigma\right)\\
    &=o\left(\|(x_i,t_i)\|^\sigma\right)\\
    &=o(1),
    	\end{split}
	\end{equation*}
    a contradiction.    
    
    \textit{Step 2.} For any given $(y,\tau)\in\overline{S}_{t_0}$, one can choose $k\in\mathbb{N}$ big enough to guarantee 
    $$
    \left(k^{-1}(y-x_0),k^{-\sigma}(\tau-t_0)\right)\in G_{1/5}.
    $$
    Using Theorem \eqref{growth}, we then estimate
    \begin{equation*}
    	\begin{split}
    		u(y,\tau)&=k^\sigma u_k\left(k^{-1}(y-x_0),k^{-\sigma}(\tau-t_0)\right)\\
    		&\le C\|u_k\|_\infty\|\left(y-x_0,\tau-t_0\right)\|^\sigma.
    	\end{split}
    \end{equation*}	
    The latter with \eqref{bounded} implies $u(y,\tau)=0$.    
\end{proof}
As a consequence of Theorem \ref{Liouville}, we obtain the next result, which reveals that for the conclusion to hold for the homogeneous equation, it is enough to assume the sharp asymptotic growth only in the space variable.
\begin{theorem}\label{t3.2}
		Let $\gamma\ge0$ be as in Theorem \ref{growth}, and $u\ge0$ be a solution of
	$$
	u_t-\Div(u^\gamma\nabla u)=0
	$$
	in a strip $S_T$, $T\in\R$. If $u(x_0,t_0)=0$ for some $(x_0,t_0)\in S_T$, and
	\begin{equation}\label{spacegrowth}
		u(x,t)=o\left(|x|^{\frac{2}{\gamma+1}}\right),\,\,\,\text{ as }\,\,\,|x|\rightarrow\infty,
	\end{equation}
	uniformly in $t\in(-\infty,T)$, then $u\equiv0$ in $\overline{S}_{t_0}$.
\end{theorem}
\begin{proof}
	Observe that \eqref{intrinsicgrowth} holds. Indeed, if that is not the case, then for a sequence $(x_i,t_i)$ and a constant $c_*>0$, one has
	\begin{equation}\label{contradiction}
		\frac{u(x_i,t_i)}{\|(x_i,t_i)\|^{\frac{2}{\gamma+1}}}\ge c_*>0,\,\,\,\textrm{ as }\,\,\,\|(x_i,t_i)\|\to\infty.
	\end{equation}
	As $\|(x_i,t_i)\|\to\infty$, then either $|x_i|\to\infty$ or $|x_i|<r$ for some $0<r<\infty$, and $t_i\to-\infty$. In the first case, combining \eqref{spacegrowth} with \eqref{contradiction}, we get a contradiction, since
	$$
	\frac{u(x_i,t_i)}{|x_i|^{\frac{2}{\gamma+1}}}\ge\frac{u(x_i,t_i)}{\|(x_i,t_i)\|^{\frac{2}{\gamma+1}}}.
	$$	
	In the second case, that is, when $|x_i|<r$ and $t_i\to-\infty$, we choose $\tau\in\R$ such that 
	$$
	\tau\le t_1-\left(\frac{c}{u(x_1,t_1)}\right)^\gamma(r+|x_1|)^2,
	$$	
	where $c>0$ is the universal constant from Proposition \ref{dibenedetto}. Since
	$$
	\tau-t_1\le-\left(\frac{c}{u(x_1,t_1)}\right)^\gamma|x-x_1|^2,\,\,\,\forall x\in\overline{B}_r,
	$$ 
	the set
	$\overline{B}_r\times\{\tau\}$ is contained in the paraboloid
	$$
	P(x_1,t_1):=\left\{(x,t)\in S_T;\,\,t-t_1\le-\left(\frac{c}{u(x_1,t_1)}\right)^\gamma|x-x_1|^2\right\}.
	$$	
	By Proposition \ref{dibenedetto},
	$$
	\lim_{t\rightarrow-\infty}u(x,t)=0,\,\,\,\textrm{for all}\,\,\,x\in\R^n,
	$$
	and the convergence is uniform in $x\in\overline{B}_r$. In particular,
	$$
	u(x_i,t_i)<1,
	$$
	for $i\in\mathbb{N}$ large enough, which contradicts to \eqref{contradiction}. Hence, \eqref{intrinsicgrowth} is true, and Theorem \ref{Liouville} implies $u\equiv0$ in $\overline{S}_{t_0}$.
\end{proof}

\medskip

\textbf{Acknowledgments.} DJA is partially supported by CNPq 311138/2019-5 and grant 2019/0014 Paraiba State Research Foundation (FAPESQ).

RT is partially supported by FCT - Funda\c{c}\~ao para a Ci\^encia e a Tecnologia, I.P., through projects PTDC/MAT-PUR/28686/2017 and UTAP-EXPL/MAT/0017/2017, as well as by the Centre for Mathematics of the University of Coimbra - UIDB/00324/2020, funded by the Portuguese Government through FCT/MCTES.

\end{document}